\documentclass[11pt]{article}

\usepackage[T1]{fontenc}
\usepackage{lmodern}
\usepackage{microtype}
\usepackage{amsmath,amssymb,amsthm,mathtools}
\usepackage[margin=1.15in]{geometry}
\usepackage{enumitem}
\usepackage[hidelinks]{hyperref}

\newtheorem{theorem}{Theorem}[section]

\newtheorem{proposition}[theorem]{Proposition}
\newtheorem{corollary}[theorem]{Corollary}
\newtheorem{remark}[theorem]{Remark}

\newcommand{\F}{\mathbb F}
\newcommand{\Z}{\mathbb Z}
\newcommand{\Nul}{\mathcal N}
\newcommand{\Jac}{\operatorname{J}}
\newcommand{\Mat}{\operatorname{M}}
\newcommand{\Tri}{\operatorname{T}}
\newcommand{\cR}{\mathcal R}

\title{Peirce Stability of Null Polynomials and a Sharp Radical Bound for Finite Rings}
\author{Hongfeng Wu\\
College of Science, North China University of Technology, Beijing, China\\
whfmath@gmail.com}
\date{}

\begin{document}

\maketitle

\begin{abstract}
Let $R$ be a finite associative ring with identity, and let $\Nul(R)$ denote
the set of polynomials in a central indeterminate that vanish identically on
$R$ under right evaluation.  We establish two results concerning the failure
of $\Nul(R)$ to be a right ideal.  First, if $e$ is an idempotent, $f=1-e$,
and either $eRf=0$ or $fRe=0$, then right multiplication by $e$ preserves
$\Nul(R)$.  Thus both off-diagonal Peirce components must be nonzero whenever
$e$ witnesses a failure of right stability.  Second, if $J$ is a finite
nilpotent ideal whose additive group is a $2$-group and $e,f$ are complementary
idempotents such that
\[
 eJf\neq0,\qquad fJe\neq0,\qquad J^3\neq0,
\]
then $|J|\geq32$.  Werner's theorem that $\Jac(R)^3=0$ implies the
two-sidedness of $\Nul(R)$ therefore yields $|R|\geq128$ whenever $\Nul(R)$
is not two-sided.  We also construct a tiled matrix ring of characteristic
$4$ and order $128$ whose null ideal is not two-sided.
\end{abstract}



\section{Introduction}

Throughout, all rings are associative and have an identity.  Let $R$ be a
finite ring, and let $R[x]$ be the polynomial ring in a central indeterminate.
For $F(x)=\sum_i c_i x^i$, its right value at $a\in R$ is
$F(a)=\sum_i c_i a^i$.  Set
\[
 \Nul(R)=\{F\in R[x]:F(a)=0\text{ for every }a\in R\}.
\]
The set $\Nul(R)$ is always a left ideal of $R[x]$, but it need not be a right
ideal: substitution in a polynomial with noncommuting coefficients is not, in
general, multiplicative.

Werner conjectured that $\Nul(R)$ is two-sided for every finite ring
\cite{Werner2014}.  Havlovec disproved the conjecture by constructing a ring of
order $128$ for which $\Nul(R)$ is not two-sided \cite{Havlovec2026}.
Werner subsequently proved that $\Nul(R)$ is two-sided whenever
$\Jac(R)^3=0$ and constructed counterexamples whose Jacobson radicals have
arbitrarily large nilpotency index \cite{Werner2026}.

Our first result concerns the triangular Peirce decompositions determined by
an idempotent $e$: if either $eR(1-e)$ or $(1-e)Re$ is zero, then right
multiplication by $e$ preserves $\Nul(R)$.  The second result is a sharp
cardinality bound for a nilpotent ideal $J$: if both off-diagonal Peirce
components of $J$ are nonzero and $J^3\neq0$, then $|J|\geq32$.  Combining
these results with the Wedderburn decomposition of $R/\Jac(R)$ shows that the
least possible order of a finite ring with a non-two-sided null ideal is
$128$.

The bound is attained in both characteristics $2$ and $4$.  In addition to
Havlovec's characteristic-$2$ example, we construct a tiled subring of
$\Mat_2(\Z/4\Z)$ and exhibit a null polynomial whose right multiple by an
idempotent is not null.  The two examples have the same order but different
characteristics; in particular, the characteristic-$2$ example is not a
quotient of the characteristic-$4$ example.

\section{Preliminaries}

Write $J=\Jac(R)$.  By the Artin--Wedderburn theorem, the semisimple quotient
admits a decomposition
\begin{equation}\label{eq:semisimple}
 R/J\cong\prod_{i=1}^{t}\Mat_{n_i}(K_i),
\end{equation}
where the $K_i$ are finite fields.  The identities of the simple factors lift
to pairwise orthogonal idempotents $e_1,\ldots,e_t\in R$ with
$e_1+\cdots+e_t=1$.

We use three standard facts.  First, if $U\in R^\times$ and $F\in\Nul(R)$,
then $FU\in\Nul(R)$ \cite[Lemma~3.5]{Werner2014}.  Second, every element of
$R$ is a sum of units if and only if $R/J$ has no direct factor isomorphic to
$\F_2\times\F_2$ \cite[Theorem~4.6]{Stewart1972}
\cite[Proposition~2.3]{Werner2026}.  Third, $\Nul(R)$ is two-sided if and only
if
\[
 Fe_i\in\Nul(R)
 \quad\text{for every }F\in\Nul(R)\text{ and every }i
\]
\cite[Lemma~2.4]{Werner2026}.  The cited unit-sum results also imply that each
$e_i$ corresponding to a simple factor not isomorphic to $\F_2$ is a sum of
units.  Hence, if $\Nul(R)$ is not two-sided, then $R/J$ has at least two
simple factors isomorphic to $\F_2$, and there exist an idempotent $e$ lifting
the identity of one such factor and a polynomial $F\in\Nul(R)$ such that
\begin{equation}\label{eq:witness}
 Fe\notin\Nul(R).
\end{equation}

We shall also use the primary decomposition of a finite ring.  If
$R=R_1\times\cdots\times R_s$, then right evaluation is componentwise, and
\begin{equation}\label{eq:product-null}
 \Nul(R)=\Nul(R_1)\times\cdots\times\Nul(R_s).
\end{equation}
The additive primary components of a finite ring are two-sided ideals and form
such a direct product.  Since the null ideal of every finite ring of odd order
is two-sided \cite[Theorem~3.7]{Werner2014}, any ring with a non-two-sided null
ideal has a $2$-primary direct factor with the same property.

\section{Right stability in triangular Peirce decompositions}

\begin{theorem}\label{thm:triangular}
Let $R$ be a ring, let $e^2=e$, and put $f=1-e$.  If either $eRf=0$ or
$fRe=0$, then
\[
 Fe\in\Nul(R)\qquad\text{for every }F\in\Nul(R).
\]
\end{theorem}

\begin{proof}
Assume first that $fRe=0$.  For an arbitrary $r\in R$, Peirce block notation
gives
\[
 r=\begin{pmatrix}a&m\\0&b\end{pmatrix},
 \qquad
 c_i=\begin{pmatrix}\alpha_i&\mu_i\\0&\beta_i\end{pmatrix},
 \qquad
 F(x)=\sum_{i=0}^{d}c_ix^i.
\]
Here $a,\alpha_i\in eRe$, $b,\beta_i\in fRf$, and
$m,\mu_i\in eRf$.
For $i\geq1$, set
\[
 q_i=\sum_{j=0}^{i-1}a^jmb^{\,i-1-j},
 \qquad q_0=0.
\]
Then
\[
 r^i=\begin{pmatrix}a^i&q_i\\0&b^i\end{pmatrix}.
\]
Since $F$ vanishes at
$r_0=\begin{psmallmatrix}a&0\\0&b\end{psmallmatrix}$, its upper-right block
gives
\begin{equation}\label{eq:offdiag-coeff}
 \sum_i\mu_i b^i=0.
\end{equation}
The upper-left and upper-right blocks of $F(r)=0$ are
\[
 \sum_i\alpha_i a^i=0,
 \qquad
 \sum_i(\alpha_iq_i+\mu_i b^i)=0.
\]
By \eqref{eq:offdiag-coeff}, $\sum_i\alpha_iq_i=0$.  Since
\[
 c_i e=\begin{pmatrix}\alpha_i&0\\0&0\end{pmatrix},
\]
it follows that
\[
 (Fe)(r)=
 \begin{pmatrix}
  \sum_i\alpha_i a^i&\sum_i\alpha_iq_i\\0&0
 \end{pmatrix}=0.
\]

Now suppose that $eRf=0$.  For arbitrary $r\in R$, write
\[
 r=\begin{pmatrix}a&0\\n&b\end{pmatrix},
 \qquad
 c_i=\begin{pmatrix}\alpha_i&0\\\nu_i&\beta_i\end{pmatrix}.
\]
Here $a,\alpha_i\in eRe$, $b,\beta_i\in fRf$, and
$n,\nu_i\in fRe$.  Evaluating $F$ at the diagonal element
$r_0=\begin{psmallmatrix}a&0\\0&b\end{psmallmatrix}$ gives
\[
 \sum_i\alpha_i a^i=0,
 \qquad
 \sum_i\nu_i a^i=0.
\]
As
\[
 c_i e=\begin{pmatrix}\alpha_i&0\\\nu_i&0\end{pmatrix},
\]
we have
\[
 (Fe)(r)=
 \begin{pmatrix}
  \sum_i\alpha_i a^i&0\\
  \sum_i\nu_i a^i&0
 \end{pmatrix}=0
\]
for every $r\in R$.
\end{proof}

\begin{corollary}\label{cor:bidirectional}
If $e$ and $F$ satisfy \eqref{eq:witness}, then
\[
 eR(1-e)\neq0
 \qquad\text{and}\qquad
 (1-e)Re\neq0.
\]
\end{corollary}

\begin{proof}
If either displayed Peirce component were zero, Theorem~\ref{thm:triangular}
would give $Fe\in\Nul(R)$, contrary to \eqref{eq:witness}.
\end{proof}

\section{A sharp cardinality bound}

\begin{theorem}\label{thm:cardinality}
Let $J$ be a finite nilpotent ideal of a ring $R$, and suppose that its
additive group is a $2$-group.  Let $e^2=e$, put $f=1-e$, and set
\[
 A=eJe,\qquad M=eJf,\qquad N=fJe,\qquad B=fJf.
\]
If
\[
 M\neq0,\qquad N\neq0,\qquad J^3\neq0,
\]
then
\[
 |J|=|A|\,|M|\,|N|\,|B|\geq32.
\]
\end{theorem}

\begin{proof}
The product formula follows from the additive Peirce decomposition
$J=A\oplus M\oplus N\oplus B$.  Since $J^3\neq0$, choose
$x_0,y_0,z_0\in J$ such that $x_0y_0z_0\neq0$.  Expanding each factor into
its Peirce components shows that at least one summand of the expanded product
is nonzero.  Hence there are indices $i_0,i_1,i_2,i_3\in\{e,f\}$ and elements
\[
 x\in i_0Ji_1,\qquad y\in i_1Ji_2,
 \qquad z\in i_2Ji_3
\]
such that $xyz\neq0$.  Thus $xy$ and $yz$ are nonzero.

We shall repeatedly use the following observation.  If $a,t\in J$ and $t,at$
are nonzero elements of the same Peirce component, then $at\neq t$.  Indeed,
$at=t$ would imply $a^kt=t$ for every $k\geq1$, whereas $a^k=0$ for all
sufficiently large $k$ because $J$ is nilpotent.  The same argument applies to
$t$ and $ta$.  Each Peirce component is an additive subgroup of the finite
$2$-group $J$.  It follows that a component containing two distinct nonzero
elements has order at least $4$.

Encode $e$ by $0$ and $f$ by $1$.  Interchanging $e$ and $f$ complements the
index word $i_0i_1i_2i_3$, whereas passage to the opposite ring reverses it.
The hypotheses and the required cardinality bound are invariant under both
operations.  The complement-reversal orbits in $\{0,1\}^4$ therefore have the
following six representatives.

\begin{enumerate}[
 label=\textup{(\roman*)},leftmargin=*,
 itemsep=.35\baselineskip,topsep=.4\baselineskip,parsep=0pt
]
\item For \(0000\), one has $A^3\neq0$.  The chain
$A\supsetneq A^2\supsetneq A^3\supsetneq0$ is strict.  Indeed, $A=A^2$ would
give $A=A^k$ for every $k\geq1$, while $A^2=A^3$ would give $A^2=A^k$ for
every $k\geq2$; either equality contradicts nilpotence and $A^3\neq0$.
Hence $|A|\geq8$.
Since $M$ and $N$ are nonzero,
$|M|,|N|\geq2$, and therefore $|J|\geq8\cdot2\cdot2=32$.

\item For \(0001\), one has $x,y\in A$ and $z\in M$.  The elements $y$ and
$xy$ are distinct and nonzero, so $|A|\geq4$.  Moreover, $z$ and $yz$ are
distinct nonzero elements of $M$, so $|M|\geq4$.  Together with
$|N|\geq2$, this gives $|J|\geq32$.

\item For \(0010\), one has $x\in A$, $y\in M$, and $z\in N$.  The elements
$y$ and $xy$ are distinct and nonzero, hence $|M|\geq4$.  The elements
$x$, $yz$, and $x(yz)=xyz$ are nonzero and lie in $A$.  If $|A|=2$, then
$A$ has a unique nonzero element, so all three elements coincide.  It follows
that $x=yz=x(yz)=x^2$, and hence $x$ is a nonzero idempotent of the nilpotent
ring $J$, a contradiction.  Thus $|A|\geq4$, and $|J|\geq32$.

\item For \(0011\), one has $x\in A$, $y\in M$, and $z\in B$.  The elements
$y$ and $xy$ are distinct and nonzero, so $|M|\geq4$.  Since $A$, $B$, and
$N$ are nonzero, $|J|\geq2\cdot4\cdot2\cdot2=32$.

\item For \(0101\), one has $x,z\in M$ and $y\in N$.  The components
$A$, $B$, and $N$ are nonzero because they contain $xy$, $yz$, and $y$,
respectively.  If $|M|=2$, then $M$ has a unique nonzero element, so $x$,
$z$, and $xyz$ coincide; denote their common value by $m$.  With $a=xy\in A$,
we obtain $am=(xy)m=xyz=m$, contrary to the preceding observation.  Hence
$|M|\geq4$ and $|J|\geq32$.

\item For \(0110\), one has $x\in M$, $y\in B$, and $z\in N$.  The pairs
$x,xy$ and $z,yz$ consist of distinct nonzero elements, so
$|M|,|N|\geq4$.  Moreover, $B\neq0$ because $y\in B$, whereas $A\neq0$
because $xyz\in A$.  Consequently,
$|J|\geq2\cdot4\cdot4\cdot2=64$.
\end{enumerate}
These six cases exhaust the complement-reversal orbits of the index words.
\end{proof}

\section{The order bound and examples attaining it}

\begin{theorem}\label{thm:order}
If $R$ is a finite ring and $\Nul(R)$ is not a two-sided ideal of $R[x]$, then
\[
 |R|\geq128.
\]
\end{theorem}

\begin{proof}
By \eqref{eq:product-null}, $R$ has a $2$-primary direct factor whose null
ideal is not two-sided.  The order of this factor does not exceed $|R|$, so it
suffices to treat the case in which the additive group of $R$, and hence that
of $J=\Jac(R)$, is a $2$-group.  The results recalled in Section~2 show that
$R/J$ has at least two simple factors isomorphic to $\F_2$; thus $|R/J|\geq4$.
They also provide $e$ and $F$ satisfying \eqref{eq:witness}.  Put $f=1-e$.  By
Corollary~\ref{cor:bidirectional},
\[
 eRf\neq0,\qquad fRe\neq0.
\]
The image of $e$ is the identity of a direct factor of $R/J$ and is therefore
central.  It follows that $eRf\subseteq J$ and $fRe\subseteq J$, whence
$eRf=eJf$ and $fRe=fJe$.  Because $\Nul(R)$ is not two-sided, Werner's theorem
\cite[Theorem~1.2]{Werner2026} implies that $J^3\neq0$.
Theorem~\ref{thm:cardinality} now gives $|J|\geq32$, and
therefore
\[
 |R|=|R/J|\,|J|\geq4\cdot32=128.
\]
\end{proof}

The bound is attained by the following characteristic-$2$ ring constructed by
Havlovec:
\[
 \cR_4=
 \left\{
 \begin{pmatrix}C&D\\0&C\end{pmatrix}:
 C\in\Tri_2(\F_2),\ D\in\Mat_2(\F_2)
 \right\},
\]
where $\Tri_2(\F_2)$ denotes the ring of upper triangular $2\times2$ matrices
over $\F_2$ \cite{Havlovec2026}.  We next construct an example of
characteristic $4$.  Let $\Z_4=\Z/4\Z$ and
\begin{equation}\label{eq:S}
 S=\left\{
 \begin{pmatrix}a&b\\c&d\end{pmatrix}\in\Mat_2(\Z_4):
 b\in2\Z_4
 \right\}.
\end{equation}
Then $|S|=128$ and $\operatorname{char}(S)=4$.

\begin{proposition}\label{prop:S}
The null ideal $\Nul(S)$ is not two-sided.
\end{proposition}

\begin{proof}
The surjective ring homomorphism
\[
 S\longrightarrow\F_2\times\F_2,
 \qquad
 \begin{pmatrix}a&b\\c&d\end{pmatrix}
 \longmapsto(\bar a,\bar d)
\]
has kernel
\[
 J=
 \left\{
 \begin{pmatrix}2\alpha&2\mu\\\gamma&2\delta\end{pmatrix}:
 \alpha,\mu,\delta\in\F_2,\ \gamma\in\Z_4
 \right\}.
\]
Set
\[
 p=2E_{11},\quad q=2E_{22},\quad
 u=2E_{12},\quad v=E_{21},\quad w=2E_{21}.
\]
Then, as an additive group,
\[
 J=\langle p\rangle\oplus\langle q\rangle
   \oplus\langle u\rangle\oplus\langle v\rangle,
 \qquad \operatorname{ord}(p)=\operatorname{ord}(q)
 =\operatorname{ord}(u)=2,\quad \operatorname{ord}(v)=4.
\]
Direct multiplication in $S$ gives
\[
 uv=p,\qquad vu=q,\qquad qv=vp=w,
\]
and all other products among $p,q,u,v$ are zero.  Consequently,
\[
 J^2=\langle p,q,w\rangle_{\F_2},
 \qquad J^3=\langle w\rangle_{\F_2},
 \qquad J^4=0.
\]
Since $J$ is nilpotent, $J\subseteq\Jac(S)$.  On the other hand,
$S/J\cong\F_2\times\F_2$ is semisimple, so $\Jac(S)\subseteq J$.  Therefore
$J=\Jac(S)$.

Put
\[
 P(x)=x^2+x,\qquad \beta=E_{21},
 \qquad H(x)=P(x)^3-\beta P(x)^2.
\]
For $A\in S$, its image in $S/J\cong\F_2\times\F_2$ is idempotent.
Consequently, $Y=A^2+A$ belongs to $J$ and has the form
\[
 Y=\begin{pmatrix}2\alpha&2\mu\\\gamma&2\delta\end{pmatrix}.
\]
Direct multiplication modulo $4$ gives
\[
 Y^2=
 \begin{pmatrix}
  2\mu\gamma&0\\
  2\gamma(\alpha+\delta)&2\mu\gamma
 \end{pmatrix},\quad
 Y^3=\begin{pmatrix}0&0\\2\mu\gamma^2&0\end{pmatrix},\quad
 \beta Y^2=\begin{pmatrix}0&0\\2\mu\gamma&0\end{pmatrix}.
\]
Since $\gamma^2\equiv\gamma\pmod2$, these formulas give
$Y^3=\beta Y^2$.  Because the coefficients of $P$ are central,
$(P(x)^k)(A)=P(A)^k=Y^k$ for $k=2,3$.  Hence
$H(A)=Y^3-\beta Y^2=0$, and therefore $H\in\Nul(S)$.

Let
\[
 e=E_{22},\qquad r=2E_{12}+E_{21}.
\]
Then $\beta e=0$, $r^2=2I$, and, for $\theta=P(r)=2I+r$,
\[
 \theta^2=2I,\qquad \theta^3=2E_{21}\neq0.
\]
Since $P(x)$ has central coefficients, it commutes with $e$ in $S[x]$.
Together with $\beta e=0$, this gives
\[
 H(x)e=eP(x)^3,
 \]
and hence
\[
 (He)(r)=e\theta^3=2E_{21}\neq0.
\]
Thus $He\notin\Nul(S)$.
\end{proof}

The preceding computation also shows that the bound in
Theorem~\ref{thm:cardinality} is sharp.  Indeed, for $e=E_{22}$ and
$f=E_{11}$, both
$eJf$ and $fJe$ are nonzero, while $J^3\neq0$ and $|J|=32$.

\begin{corollary}\label{cor:noquotient}
No quotient of $S$ is isomorphic to $\cR_4$.
\end{corollary}

\begin{proof}
Both rings have order $128$.  If $S/I\cong\cR_4$, equality of orders forces
$I=0$, so $S\cong\cR_4$.  This is impossible because $S$ and $\cR_4$ have
characteristics $4$ and $2$, respectively.
\end{proof}

\begin{remark}
Theorem~\ref{thm:order} and Corollary~\ref{cor:noquotient} also settle,
respectively, Questions~1.3 and~5.12 posed in \cite{Werner2026}.
\end{remark}

\end{document}